\documentclass{amsart}
\date{July 12, 2013}
\usepackage[dvips]{graphicx} 
\usepackage{verbatim,enumerate}
\usepackage[usenames]{color}
\title[Zero mean curvature surfaces]{%
Zero mean curvature surfaces 
       in Lorentz-Minkowski $3$-space
       which change type across a 
       light-like line}

\author[Fujimori et.\ al.]{
        S.~Fujimori,  % Shoichi Fujimori
	Y.~W.~Kim,    % Young Wook Kim
	S.-E.~Koh,    % Sung-Eun Koh
        W.~Rossman,   % Wayne Rossman
	H.~Shin,      % Heayong Shin
        M.~Umehara,   % Masaaki Umehara
        K.~Yamada and % Kotaro Yamada
	S.-D.~Yang}   % Seong-Deog Yang

\address[Fujimori]{%
   Department of Mathematics,
   Faculty of Science, Okayama University,
   Okayama 700-8530, Japan}
\email{fujimori@math.okayama-u.ac.jp}
\address[Kim]{%
   Department of Mathematics,
   Korea University, Seoul 136-701, Korea
}
\email{ywkim@korea.ac.kr}
\address[Koh]{%
   Department of Mathematics,
   Konkuk University, Seoul 143-701, Korea}
\email{sekoh@konkuk.ac.kr}
\address[Rossman]{%
   Department of Mathematics,
   Faculty of Science,
   Kobe University,
   Kobe 657-8501, Japan
}
\email{wayne@math.kobe-u.ac.jp}
\address[Shin]{%
   Department of Mathematics,
   Chung-Ang University, Seoul 156-756, Korea
}

\email{hshin@cau.ac.kr}
\address[Umehara]{%
   Department of Mathematical and Computing Sciences,
   Tokyo Institute of Technology,
   Tokyo 152-8552, Japan
}
\email{umehara@is.titech.ac.jp}

\address[Yamada]{%
   Department of Mathematics,
   Tokyo Institute of Technology,
   Tokyo 152-8551, Japan
}
\email{kotaro@math.titech.ac.jp}
\address[Yang]{%
   Department of Mathematics,
   Korea University, Seoul 136-701, Korea
}
\email{sdyang@korea.ac.kr}

\usepackage{amsthm}
\theoremstyle{plain}
 \newtheorem{theorem}{Theorem}[section]
 \newtheorem{proposition}[theorem]{Proposition}

 \newtheorem{lemma}[theorem]{Lemma}
 \newtheorem{corollary}[theorem]{Corollary}
 
\theoremstyle{definition}
 
\theoremstyle{remark}
 
 \newtheorem*{remark*}{Remark}

\numberwithin{equation}{section}

\newcommand{\R}{\boldsymbol{R}}

\renewcommand{\phi}{\varphi}

\begin{document}
\begin{abstract}
 It is well-known that space-like maximal surfaces 
 and time-like minimal surfaces in Lorentz-Minkowski $3$-space 
 $\R^3_1$  have singularities in  general. 
 They are both characterized as zero mean curvature surfaces.
 We are interested in the case where the singular set
 consists of a light-like line,
 since this case has not been analyzed before.
 As a continuation of a previous work by the authors,
 we give the first example 
 of a family of such surfaces which change type across the light-like line.
 As a corollary, we also obtain a family of zero mean curvature
 hypersurfaces in $\R^{n+1}_1$ that change type across an 
 $(n-1)$-dimensional light-like plane.
\end{abstract}
\thanks{
  Kim was supported by NRF 2009-0086794, 
  Koh by NRF 2009-0086794 and NRF 2011-0001565, 
  and Yang by NRF 2012R1A1A2042530.
  Fujimori was partially supported by the Grant-in-Aid for 
  Young Scientists (B) No.\ 21740052,
  Rossman was supported by 
  Grant-in-Aid for Scientific Research (B) No.\ 20340012, 
  Umehara by (A) No.\ 22244006 and 
  Yamada by (B) No.\ 21340016
  from Japan Society for the Promotion of Science.
}
\maketitle

\section*{Introduction}
Many examples of space-like maximal surfaces 
containing singular curves in
the Lorentz-Minkowski $3$-space
$(\R^3_1;t,x,y)$ of signature $(-++)$  have been constructed 
in \cite{K}, \cite{ER}, \cite{UY}, \cite{KY1} and \cite{FRUYY}.

In this paper, we are interested in 
the zero mean curvature surfaces
in $\R^3_1$ changing their causal type:
Klyachin \cite{Kl} showed under a 
sufficiently weak regularity assumption that
a zero mean curvature surface
in $\R^3_1$ changes its causal type
only on the following two subsets:
\begin{itemize} 
\item null curves
      (i.e., regular curves whose velocity vector fields are light-like)
      which are non-degenerate (i.e., their projections
into the $xy$-plane are locally convex plane curves), or
 \item light-like lines, which are degenerate everywhere.
\end{itemize} 

Given a non-degenerate 
null curve $\gamma$ in $\R^3_1$, 
there exists a zero mean curvature surface which 
changes its causal type across this curve from a space-like 
maximal surface to a time-like minimal surface 
(cf.\ \cite{G}, \cite{Kl}, \cite{KY2} and \cite{KKSY}).
This construction can be accomplished using 
the Bj\"orling formula for the Weierstrass-type
representation formula of maximal surfaces.

However, if $\gamma$ is a light-like line,
the aforementioned construction fails, 
since the isothermal coordinates break down
at the light-like singular points.
Locally, such a surface is the
graph of a function $t=f(x,y)$ satisfying 
\begin{equation}\tag{$*$}\label{zm}
    (1-f_y^2)f_{xx} + 2 f_x f_yf_{xy}+(1- f_x^2)f_{yy}=0,
\end{equation}
where $f_x=\partial f/\partial x$, $f_{xy}=\partial^2 f/(\partial
x\partial y)$, etc.
We call this and its graph the {\it zero mean curvature equation\/} 
and a {\it zero mean curvature surface}, respectively.
Until now, zero mean curvature
surfaces which actually change type across a light-like line
were unknown.
As announced in \cite{CR},
the main purpose of this paper is to construct such an example.
In Section~\ref{sec:existence}, we give a formal
power series solution of the zero mean
curvature equation describing
all zero mean curvature surfaces
which contain a light-like line.
Using this, we give the precise statement 
of our main result and show how
the statement can be reduced to
a proposition (cf.\ Proposition~\ref{prop:goal}). 
In Section~\ref{sec:proof}, we then prove it.
As a consequence, we obtain the first 
example of (a family of) zero mean curvature surfaces
which change type across a light-like line.

\section{%
% Existence of zero mean curvature
% surfaces changing type across a light-like line
 The Main Theorem
}\label{sec:existence}
We discuss solutions of the zero mean curvature equation 
\eqref{zm} which have the following form
\begin{equation}
 \label{eq:series}
    f(x,y)=b_0(y)+\sum_{k=1}^\infty \frac{b_k(y)}{k}x^k, 
\end{equation}
where $b_k(y)$ ($k=1,2,\dots$) are $C^\infty$-functions. 
When $f$ contains a singular light-like line,
we may assume without loss of generality that (cf.\ \cite{CR})
\begin{equation}\label{eq:binit-1}
   b_0(y)=y,\quad b_1(y)=0.
\end{equation}
As was pointed out in \cite{CR},
there exists a real constant $\mu$ called
the {\it characteristic} of $f$ such that
$b_2(y)$ satisfies the following 
equation
\begin{equation}\label{eq:b20}
   b'_2(y)+b_2(y)^2+\mu=0\qquad
   \left(~'=\frac{d}{dy}\right).
\end{equation}

Now we derive the differential equations satisfied by $b_k(y)$ for 
$k \ge 3$ assuming \eqref{eq:binit-1}.
If we set
\[
   Y:=f_y-1=\sum_{k=2}^\infty \frac{b'_k(y)}{k}x^k
\]
and
\[
  \widetilde P:=2(Yf_{xx}-f_xf_{xy}), \quad
  Q:=Y^2f_{xx}-2f_xf_{xy}Y,\quad
  R:=f_x^2 f_{yy},
\]
then, by straightforward calculations, we see that   
\begin{align*}
 &\widetilde P=-b_2 b'_2 x^2-\frac43 b_2 b'_3 x^3
   -\sum_{k=4}^\infty \left(P_k+
   \frac{2(k-1)}k b_2 b'_k+(3-k)b'_2 b_k\right)x^k,\\
 &Q=-\sum_{k=4}^\infty Q_kx^k,
 \quad
  R=\sum_{k=4}^\infty R_kx^k,
\end{align*}
where
\begin{equation}\label{eq:PQR}
\begin{aligned}
 P_k&:=\sum_{m=3}^{k-1}\frac{2(k-2m+3)}{k-m+2} b_m b'_{k-m+2}, \\
 Q_k&:=
 \sum_{m=2}^{k-2}\sum_{n=2}^{k-m}
 \frac{3n-k+m-1}{mn}b'_m b'_{n}b_{k-m-n+2}, \\
 R_k&:=
 \sum_{m=2}^{k-2}\sum_{n=2}^{k-m}
 \frac{b_m b_{n}b''_{k-m-n+2}}{k-m-n+2} 
\end{aligned}
\end{equation}
for $k\ge 4$,
and that the zero mean curvature equation \eqref{zm} reduces to 
\[
   \sum_{k=2}^\infty \frac{b''_k}k x^k=f_{yy}=\widetilde P+Q+R.
\]
It is now immediate, by comparing the coefficients of $x^k$ from both sides, to see that each $b_k$ ($k\ge 3$)
satisfies the following 
ordinary differential equation
\begin{equation}\label{eq:b2}
   b''_k(y)+2(k-1)b_2(y) b'_k(y)+k(3-k)b'_2(y)b_k(y)
                =-k(P_k+Q_k-R_k),
\end{equation}
where $P_3=Q_3=R_3=0$ 
and
$P_k$, $Q_k$ and $R_k$ are as in \eqref{eq:PQR} for $k\geq 4$.
Note that $P_k$, $Q_k$ and $R_k$
are written in terms of $b_{j}$ ($j=1,\dots,k-1$)
and their derivatives.

\medskip

Now, we consider the case that
$1-f^2_x-f_y^2$ changes sign across the light-like line 
$\{t=y,x=0\}$.
This case occurs only when the characteristic $\mu$ 
as in \eqref{eq:b20} of $f$ vanishes \cite{CR}.
If we set 
\[
   b_2(y)=0\qquad (y\in \R),
\]
then \eqref{eq:b20} holds for $\mu=0$.
So we assume 
\begin{equation}\label{eq:initial-b}
   b_0(y)=y,\quad
   b_1(y)=0,\quad
   b_2(y)=0,\quad
   b_3(y)=3c y,
\end{equation}
where $c$ is a non-zero constant.
Then $f(x,y)$ in \eqref{eq:series}
can be rewritten as 
\begin{equation}\label{eq:series-red}
   f(x,y)=y+c yx^3+\sum_{k=4}^\infty \frac{b_k(y)}{k}x^k.
\end{equation}
In this situation, we will find a solution satisfying
\begin{equation}\label{eq:initial-k}
   b_k(0)=b'_k(0)=0\qquad (k\ge 4).
\end{equation}
Then \eqref{eq:b2} reduces to
\begin{alignat}{2}
  b''_k(y)
    &= -k (P_k+Q_k-R_k),\quad b_k(0)=b'_k(0)=0,\quad&&(k=4,5,\dots),
      \label{eq:diffeq}\\
  P_k & = \sum_{m=3}^{k-1} \frac{2(k-2m+3)}{k-m+2}b_m(y)b'_{k-m+2}(y)
  \quad &&(k \ge 4), \label{eq:pk}\\
  Q_k & = 
     \sum_{m=3}^{k-4}\sum_{n=3}^{k-m-1}
         \frac{3n-k+m-1}{mn}b'_{m}(y)b'_{n}(y)b_{k-m-n+2}(y)
  ~&&(k \ge 7), \label{eq:qk}\\
  R_k & =
     \sum_{m=3}^{k-4}\sum_{n=3}^{k-m-1}
         \frac{b_m(y)b_n(y)b''_{k-m-n+2}(y)}{k-m-n+2}
  \quad &&(k \ge 7), \label{eq:rk}
\end{alignat}
and $Q_k=R_k=0$ for 
$ 4 \le k\le 6$, 
where the fact that $b_2(y)=0$ has been extensively used.
For example,
\begin{align*}
&b_0=y,\quad b_1=b_2=0,\quad b_3=3cy,\quad b_4=-4c^2y^3,\quad
b_5=9 c^3 y^5,\\
&b_6=-24c^2y^7,\quad
b_7={70 c^5 y^9}-14 c^3 y^3,\,\, \dots\quad.
\end{align*}
\begin{figure}[t]%
 \begin{center}
       \includegraphics[width=7.8cm]{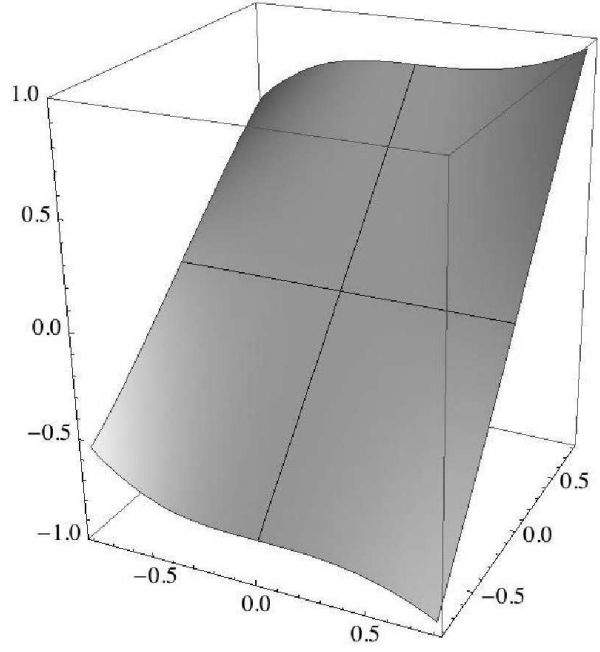}  
 \end{center}
\caption{%
 The graph of $f(x,y)$ for $c=1/2$ and $|x|,|y|<0.8$
 (The range of the graph is wider than the range used in 
 our mathematical estimation. However, this figure 
 still has a sufficiently small numerical error term 
 in the Taylor expansion.)
}%
\label{fig2}
\end{figure}

In this article, we show the following assertion:
\begin{theorem}\label{thm:main}
 For each positive number $c$,
 the formal power series solution $f(x,y)$
 uniquely determined by \eqref{eq:diffeq}, \eqref{eq:pk},
 \eqref{eq:qk} and \eqref{eq:rk}
 gives a real analytic zero mean curvature
 surface on a neighborhood of $(x,y)=(0,0)$.
 In particular, there exists a non-trivial
 $1$-parameter family of real analytic zero mean curvature
 surfaces each of which changes type 
 across a light-like line {\rm (see Figure\ \ref{fig2})}.
\end{theorem}

As a consequence, we get the following:

\begin{corollary}
 There exists a family of zero-mean curvature hypersurfaces in
 Lorentz-Minkowski space $\R^{n+1}_1$
 which change type across an
 $(n-1)$-dimensional 
 light-like plane.
\end{corollary}
\begin{proof}
 Let $f$ be as in the theorem.
 The graph of the function defined by
 \[
     \R^{n}\ni (x_1,\dots,x_n)\longmapsto f(x_1,x_2)\in \R
 \]
 gives the desired hypersurface.
 In this case, the zero mean curvature equation
 \[
    \left( 1-\sum_{j=1}^nf^2_{x_j}\right)\sum_{i=1}^n f_{x_i,x_i}+
     \sum_{i,j=1}^n f_{x_ix_j}f_{x_i}f_{x_j}=0
     \quad
    \left(f_{x_i}:=\frac{\partial}{\partial x_i},~
        f_{x_i,x_j}:=\frac{\partial^2f}{\partial x_i\partial x_j}\right)
 \]
reduces to \eqref{zm} in the introduction.
\end{proof}

To prove Theorem~\ref{thm:main}, 
it is sufficient to show that for arbitrary positive constants 
$c>0$ and $\delta >0$ there exist positive constants 
$n_0$, $\theta_0$, 
and $C$ 
such that
\begin{equation}\label{eq:W}
|b_k(y)|\le \theta_0 C^{k}\qquad (|y| 
\le \delta)
\end{equation}
holds for $k\ge n_0$. 
In fact, if \eqref{eq:W} holds,
then the series \eqref{eq:series-red} converges uniformly
over the rectangle $[-C^{-1},C^{-1}]\times [-\delta,\delta]$.

The key assertion to prove \eqref{eq:W} is the following
\begin{proposition}\label{prop:goal}
 For each $c>0$ and $\delta>0$, we set
 \begin{equation}\label{eq:M}
   M:= 3\max\left\{ 144\,c\,\tau |\delta|^{3/2},\,
              \sqrt[4]{192 c^2 \tau}
       \right\},
 \end{equation}
 where $\tau$ is the positive constant 
 given by \eqref{eq:tau2} in the appendix,
 such that
 \begin{equation}\label{eq:tau}
     t \int_{t}^{1-t}\frac{du}{u^2(1-u)^2}
 \leq \tau\qquad \left(0<t<\frac{1}{2}\right).
 \end{equation}

 Then the function $\{b_l(y)\}_{l\ge 3}$ 
 formally determined by
 the recursive formulas
 \eqref{eq:diffeq}--\eqref{eq:rk}
 satisfies the inequalities
 \begin{align}
   |b''_l(y)|&\leq c|y|^{l^*}M^{l-3},\label{eq:b-est-1}\\
   |b'_l(y)|&\leq \frac{3c|y|^{l^*+1}}{l^*+2}M^{l-3},\label{eq:b-est-2}\\
   |b_l(y)|&\leq \frac{3c|y|^{l^*+2}}{(l^*+2)^2}M^{l-3} \label{eq:b-est-3}
 \end{align}
 for any 
 \begin{equation}\label{eq:y-range}
    y\in [-\delta,\delta],
 \end{equation}
 where
 \begin{equation}\label{eq:l-star}
    l^*:=\frac{1}{2}(l-1)-2\qquad (l=3,4,\dots).
 \end{equation}
\end{proposition}
Once this proposition is proven, \eqref{eq:W} 
follows immediately. In fact, if we set
\[
   \theta_0=\frac 3c (\delta M)^3,\qquad 
   C:= \delta M
\]
and $n_0\ge 7$,
then $1\le l^*+2<l-3$ and \eqref{eq:W} follows from
\[
   \frac{3c|y|^{l^*+2}}{(l^*+2)^2}M^{l-3}\le 
   \theta_0C^{l}.
\]

\section{Proof of Proposition \ref{prop:goal}}\label{sec:proof}
We prove the proposition using induction on
the number $l\ge 3$.
If $l=3$, then
\begin{align*}
   |b''_3(y)| &= 0 \leq \frac{c}{|y|}= c|y|^{3^*}M^{0},\\
   |b'_3(y)|  &= 3c = \frac{3c|y|^{3^*+1}}{3^*+2}M^0,\\
   |b_3(y)|     &= 3c|y| = \frac{3c|y|^{3^*+2}}{(3^*+2)^2}M^0
\end{align*}
hold, using that $b_3(y)=3c y$, $M^0=1$ 
and $3^*=-1$.
So we prove the assertion for $l\ge 4$.
Since \eqref{eq:b-est-2}, \eqref{eq:b-est-3} 
follow from \eqref{eq:b-est-1} by integration,
it is sufficient to show that
\eqref{eq:b-est-1} holds for each $l\ge 4$.
(In fact, 
the most delicate case is $l=4$.
In this case $l^*=-1/2$ and we can use the 
fact that $\int_0^{y_0} 1/\sqrt{y} \,dy$
for $y_0>0$ converges.)

The inequality \eqref{eq:b-est-1} 
follows if one shows that, for each $k \ge 4$
\begin{equation}\label{eq:concl}
   \bigl|k P_k(y)\bigr|,\,\,
   \bigl|k Q_k(y)\bigr|,\,\,
   \bigl|k R_k(y)\bigr|\,\,\leq \frac{c}{3} |y|^{k^*}M^{k-3}
\quad (|y|\le \delta)
\end{equation}
under the assumption that 
\eqref{eq:b-est-1}, \eqref{eq:b-est-2} and \eqref{eq:b-est-3}
hold for all $3 \le l \le k-1$.
In fact, if \eqref{eq:concl} holds,
\eqref{eq:b-est-1}  for $l=k$ follows immediately.
Then by the initial condition \eqref{eq:diffeq} 
(cf.\ \eqref{eq:initial-k}),
we have \eqref{eq:b-est-2} and \eqref{eq:b-est-3}
for $l=k$ by integration.

\subsection*{The estimation of $|kP_k|$ for $k \ge 4$.}
By \eqref{eq:pk} 
and using the fact that \eqref{eq:b-est-2},
\eqref{eq:b-est-3} 
hold for 
$l\leq k-1$, we have 
for each $|y|<\delta$ that
\allowdisplaybreaks{%
\begin{alignat*}{2}
 |kP_k|&\leq \sum_{m=3}^{k-1}
         \frac{2k|k-2m+3|}{k-m+2}
           \bigl|b_m(y)\bigr|\,
           \bigl|b'_{k-m+2}(y)\bigr|\\
       &\le \sum_{m=3}^{k-1}
             \frac{2k|k-2m+3|}{k-m+2}
             \left(\frac{3c M^{m-3}|y|^{m^*+2}}{(m^*+2)^2}\right)
             \left(\frac{3c M^{k-m+2-3}|y|^{(k-m+2)^*+1}}{(k-m+2)^*+2}\right)
\\
       &= c M^{k-3}|y|^{k^*}\frac{144 c |y|^{\frac{3}{2}}}{M}
             \sum_{m=3}^{k-1}
             \frac{k|k-2m+3|}{(m-1)^2(k-m+1)(k-m+2)}\\
       &\le c M^{k-3}|y|^{k^*}\frac{144 c |\delta|^{\frac{3}{2}}}{M}
             \sum_{m=3}^{k-1}
             \frac{k|k-2m+3|}{(m-1)^2(k-m+1)(k-m+2)}\\
       &\le \frac{c}{3\tau} M^{k-3}|y|^{k^*}
             \sum_{m=3}^{k-1}
             \frac{k|k-2m+3|}{(m-1)^2(k-m+1)^2}.
\intertext{%
Here, we used \eqref{eq:M}.
Since}
&\max_{m=3,\dots,k-1}|k-2m+3|=\max_{m=3,k-1}
 |k-2m+3|=\max\{|k-3|,|-k+5|\},
\intertext{
by setting $q=m-1$, we have that
}
  |kP_k| 
       &\leq \frac{c}{3\tau} M^{k-3}|y|^{k^*}
             \sum_{m=3}^{k-1}
             \frac{k^2}{(m-1)^2(k-m+1)^2}%\\
       = \frac{c}{3\tau} M^{k-3}|y|^{k^*}
             \frac{1}{k}
             \sum_{q=2}^{k-2}
             \frac{k^3}{q^2(k-q)^2} \\
& \leq \frac{c}{3\tau} M^{k-3}|y|^{k^*}
             \frac{1}{k}
             \int_{\frac{1}{k}}^{1-\frac{1}{k}}
                \frac{du}{u^2(1-u)^2}
\leq \frac{c}{3} M^{k-3}|y|^{k^*},
\end{alignat*}}%
where we applied Lemma \ref{lem:sum-int}
and \eqref{eq:tau} at the last step of the estimations.
Hence, we get \eqref{eq:concl} for $k P_k$.

\subsection*{The estimation of $|kQ_k|$ for $k \ge 7$.}
By \eqref{eq:qk} and the
induction assumption, we have that
\begin{align*}
 |kQ_k|&\leq \sum_{m=3}^{k-4}\sum_{n=3}^{k-m-1}
         \frac{k|3n-k+m-1|}{mn}
           \bigl|b'_m(y)\bigr|\,
           \bigl|b'_n(y)\bigr|\,
           \bigl|b_{k-m-n+2}(y)\bigr|\\
       &\leq \sum_{m=3}^{k-4}\sum_{n=3}^{k-m-1}
         \frac{k|3n-k+m-1|}{mn}
          \left(\frac{3c M^{m-3}|y|^{m^*+1}}{m^*+2}\right)\times\\
       &\hphantom{\leq \sum_{m=3}^{k-4}xxxxxx}
          \left(\frac{3c M^{n-3}|y|^{n^*+1}}{n^*+2}\right)
          \left(\frac{3c M^{k-m-n+2-3}|y|^{(k-m-n+2)^*+2}}
{\bigl((k-m-n+2)^*+2\bigr)^2}\right)\\
       &= c M^{k-3}|y|^{k^*}
          \frac{432c^2}{M^4}
          \sum_{m=3}^{k-4}\sum_{n=3}^{k-m-1}
          \frac{k|3n-k+m-1|}{(m-1)^2(n-1)^2(k-m-n+2)^2}.
\end{align*}
Now we apply the inequality  
\begin{align*}
 \max_{{3\le m\le k-4}\atop{3\le n\le k-m-1}} |3n-k+m-1| 
 &= 
 \max_{(m,n)=(3,3),(3,k-4),(k-4,3)}
 |3n-k+m-1|\\
 &=\max \{|-k+11|,\,4,\,|2k-10|\}\leq 2k,
\end{align*}
and also
\[
  \frac{432c^2}{M^4}\le \frac{1}{36\tau},
\]
which follows from \eqref{eq:M}.
Setting $p:=m-1$, $q=n-1$, we have that
\begin{align*}
  |kQ_k| 
       &\leq \frac{c}{36\tau}M^{k-3}|y|^{k^*}
          \sum_{m=3}^{k-4}\sum_{n=3}^{k-m-1}
          \frac{2k^2}{(m-1)^2(n-1)^2(k-m-n+2)^2}
        \\
       &= \frac{c}{18\tau} M^{k-3}|y|^{k^*}
          \sum_{p=2}^{k-5}\sum_{q=2}^{k-p-2}
          \frac{k^2}{p^2q^2(k-p-q)^2}.
\end{align*}
Now applying Lemma \ref{lem:sum-2},
we have that
\[
  |kQ_k| 
       \le  \frac{c}{18\tau} M^{k-3}|y|^{k^*}\times 6\tau
        \le \frac{c}{3} M^{k-3}|y|^{k^*},
\]
which proves \eqref{eq:concl} for $k Q_k$.

\subsection*{The estimation of $|kR_k|$  for $k \ge 7$.}
Like as in the case of $|kQ_k|$, we have
that
\begin{align*}
 |kR_k|&\leq \sum_{m=3}^{k-4}\sum_{n=3}^{k-m-1}
         \frac{k \bigl|b_m(y)\bigr|\,
           \bigl|b_n(y)\bigr|\,
           \bigl|b''_{k-m-n+2}(y)\bigr|}{k-m-n+2}\\
       &\leq \sum_{m=3}^{k-4}\sum_{n=3}^{k-m-1}
         \frac{k}{k-m-n+2}
          \left(\frac{3c M^{m-3}|y|^{m^*+2}}{(m^*+2)^2}\right)
\times\\
       &\hphantom{\leq \sum_{m=3}^{k-4}xxxx}
          \left(\frac{3c M^{n-3}|y|^{n^*+2}}{(n^*+2)^2}\right)
          \left(c M^{k-m-n+2-3}|y|^{(k-m-n+2)^*}\right)\\
       &= 144c^3 M^{k-7}|y|^{k^*}
          \sum_{m=3}^{k-4}\sum_{n=3}^{k-m-1}
          \frac{k}{(k-m-n+2)
            (m-1)^2(n-1)^2
           }\\
       &= c M^{k-3}|y|^{k^*}\frac{144 c^2}{M^4}
          \sum_{m=3}^{k-4}\sum_{n=3}^{k-m-1}
          \frac{k^2}{(k-m-n+2)^2
            (m-1)^2(n-1)^2}.
\end{align*}
Now we set $p=m-1$, $q=n-1$, 
and using the inequality
\[
  3^4 \times 144 c^2 \tau\le 3^4\times 192 c^2 \tau<M^4,
\]
we have that
\[
  |kR_k| 
       \leq \frac{c}{3^4 \tau} M^{k-3}|y|^{k^*}
           \sum_{p=2}^{k-5}\,\,\sum_{q=2}^{k-p-2}
          \frac{k^2}{p^2q^2(k-p-q)^2}.
\]
By applying  
Lemma \ref{lem:sum-2},
we have that
\[
  |kR_k| 
       \leq \frac{c}{3^4 \tau} M^{k-3}|y|^{k^*} \times 6\tau
   < 
         \frac{c}{3}M^{k-3}|y|^{k^*},
\]
which proves \eqref{eq:concl} for $k R_k$.
This completes the proof of Proposition \ref{prop:goal}.

\appendix
\section{%
 Inequalities used in the proof of Theorem~\ref{thm:main}
}\label{sec:eval}
For $a>0$, it holds that
 \begin{equation}\label{eq:partial}
    \frac{1}{u^2(a-u)^2} 
    = 
    \frac{1}{a^3}
    \left(
      \frac{a}{u^2}+\frac{2}{u} + \frac{a}{(a-u)^2}+\frac{2}{a-u}
    \right).
 \end{equation} 
Therefore,
 \begin{equation}\label{eq:int}
    \int_{t}^{a-t} \frac{du}{u^2(a-u)^2}
    =\frac{2}{a^3}\left(
               \frac{a(a-2t)}{t(a-t)}+2\log\frac{a-t}{t}
                  \right)
\qquad (0<t<\frac{a}{2}).
 \end{equation}
In particular, one can show that
there exists a positive constant $\tau$ such that
\begin{equation}\label{eq:tau2}
     t \int_{t}^{1-t}\frac{du}{u^2(1-u)^2}
 \leq \tau\qquad \left(0<t<\frac{1}{2}\right).
\end{equation}
The following assertion is needed to prove \eqref{eq:concl}
for $k P_k(y)$:

\begin{lemma}\label{lem:sum-int}
Let $p$ be a non-negative integer
and $k$ an integer satisfying
$k\geq p+4$.
Then the inequality
 \[
     \sum_{q=2}^{k-p-2}
       \frac{k^3}{q^2(k-p-q)^2}\leq 
     \int_{\frac{1}{k}}^{a-\frac{1}{k}}
            \frac{du}{u^2(a-u)^2}
     \qquad
     \left(a:=1-\frac{p}{k}\right)
 \]
holds.
\end{lemma}
\begin{proof}
In fact, if we set $a:=1-({p}/{k})$,
then \eqref{eq:partial} yields that
\allowdisplaybreaks{
 \begin{alignat*}{2}
    \frac{k^3}{q^2(k-p-q)^2}&=
    \frac{1}{k}\frac{1}{\left(\frac{q}{k}\right)^2
                        \left(a-\frac{q}{k}\right)^2}\\
         &=\frac{1}{a^3}
                    \left[
                     \frac{1}{k}
                     \left(
                        \frac{a}{\left(\frac{q}{k}\right)^2}+
                        \frac{2}{\left(\frac{q}{k}\right)}
                     \right)+
                     \frac{1}{k}
                     \left(
                        \frac{a}{\left(a-\frac{q}{k}\right)^2}+
                        \frac{2}{a-\frac{q}{k}}
                     \right)
                    \right].
\intertext{
Since 
$x\mapsto
(a+2x)/{x^2}
$
is a monotone decreasing function
and the function
$x\mapsto
(a+2(a-x))/{(a-x)^2}
$
is monotone increasing
on the interval
$(0,{a}/{2})$,
we have that}
    \frac{k^3}{q^2(k-p-q)^2}&  
\leq
         \frac{1}{a^3}
                    \left[
                     \int_{\frac{q-1}{k}}^{\frac{q}{k}}
                     \left(
                        \frac{a}{u^2}+
                        \frac{2}{u}
                     \right)du+
                   \int_{\frac{q}{k}}^{\frac{q+1}{k}}
                       \left(
                        \frac{a}{(a-u)^2}+
                        \frac{2}{a-u}
                     \right)du
                    \right],\\
\intertext{%
which yields that
}
     \sum_{q=2}^{k-p-2}
       \frac{a^3k^3}{q^2(k-p-q)^2}
    &\leq 
         \int_{\frac{1}{k}}^{a-\frac{2}{k}}
                     \left(
                        \frac{a}{u^2}+
                        \frac{2}{u}
                     \right)du+
                   \int_{\frac{2}{k}}^{a-\frac{1}{k}}
                       \left(
                        \frac{a}{(a-u)^2}+
                        \frac{2}{a-u}
                     \right)du
\\
    &\leq 
         \int_{\frac{1}{k}}^{a-\frac{1}{k}}
                     \left(
                        \frac{a}{u^2}+
                        \frac{2}{u}
                     \right)du+
                   \int_{\frac{1}{k}}^{a-\frac{1}{k}}
                       \left(
                        \frac{a}{(a-u)^2}+
                        \frac{2}{a-u}
                     \right)du
\\
    &\leq 
         \int_{\frac{1}{k}}^{a-\frac{1}{k}}
                                   \left(
                        \frac{a}{u^2}+
                        \frac{2}{u}+
                        \frac{a}{(a-u)^2}+
                        \frac{2}{a-u}
                     \right)du \\
        &=\int_{\frac{1}{k}}^{a-\frac{1}{k}}\frac{du}{u^2(a-u)^2}.
 \end{alignat*}}
This proves the assertion.
\end{proof}

The following assertion is needed to prove \eqref{eq:concl}
for $k Q_k(y)$ and $k R_k(y)$:

\begin{lemma}\label{lem:sum-2}
For any integer $k \ge 7$, the
following inequalities hold{\rm:}
 \[
     \sum_{p=2}^{k-5}
     \sum_{q=2}^{k-p-2}
       \frac{k^2}{p^2q^2(k-p-q)^2}\leq \frac{6}{k}
       \int_{\frac{1}{k}}^{1-\frac{1}{k}}\frac{du}{u^2(1-u)^2}
       \leq 6\tau,
 \]
where $\tau$ is a constant satisfying \eqref{eq:tau2}.
\end{lemma}
\begin{proof}
We set
$a = a(p):=1-({p}/{k})$.
Applying Lemma \ref{lem:sum-int}
and the identity \eqref{eq:int},
we have that
 \allowdisplaybreaks{%
 \begin{alignat*}{2}
     \sum_{p=2}^{k-5}&
     \sum_{q=2}^{k-p-2}
       \frac{k^2}{p^2q^2(k-p-q)^2}
     =\sum_{p=2}^{k-5}\left[
       \frac{1}{k p^2}
     \sum_{q=2}^{k-p-2}
       \frac{k^3}{q^2(k-p-q)^2}\right]\\
     &\leq\sum_{p=2}^{k-5}
       \left[
       \frac{1}{k p^2}
       \int_{\frac{1}{k}}^{a-\frac{1}{k}}
          \frac{du}{u^2(a-u)^2}
       \right] %\\
=     \sum_{p=2}^{k-5}
       \left[
       \frac{1}{p^2}
       \frac{2}{a^2}\left(
          \frac{a-\frac{2}{k}}{a-\frac{1}{k}}
          + 2 \frac{\log(ka-1)}{ka}
       \right)
       \right]\\
     &\leq
     \sum_{p=2}^{k-5}
       \left[
       \frac{2}{p^2 a^2}
       \left( 1 + 2\frac{\log ka}{ka}
       \right)
       \right]
     \leq
     \sum_{p=2}^{k-5}
       \frac{6}{p^2 a^2},
\end{alignat*}}%
where we used the fact that
$\frac{\log ka}{ka}<1$.
By applying Lemma \ref{lem:sum-int}
and by using the property \eqref{eq:tau2}
of the constant $\tau$,
it holds that
 \allowdisplaybreaks{%
 \begin{alignat*}{2}
     \sum_{p=2}^{k-5}
     \sum_{q=2}^{k-p-2}
       \frac{k^2}{p^2q^2(k-p-q)^2}
     &\le 6\sum_{p=2}^{k-5}
       \frac{1}{p^2 \left(1-\frac{p}{k}\right)^2}
      =\frac{6}{k}
       \sum_{p=2}^{k-5}
       \frac{k^3}{p^2 \left(k-p\right)^2}\\
     &\leq\frac{6}{k}
       \sum_{p=2}^{k-2}
       \frac{k^3}{p^2 \left(k-p\right)^2}
     \leq\frac{6}{k}
        \int_{\frac{1}{k}}^{1-\frac{1}{k}}
            \frac{du}{u^2(1-u)^2}<6\tau,
 \end{alignat*}
which proves the assertion.
}
\end{proof}

\end{document}